\documentclass[11pt,twoside,reqno,centertags,english]{amsart}
\usepackage[oldstylenums]{kpfonts}
\usepackage{amssymb}
\usepackage{amsmath}
\usepackage{mathtools}
\usepackage{amsthm}
\usepackage{bm}
 \usepackage{graphicx}
\usepackage{subcaption}
\usepackage[margin=2.5cm]{geometry}
\usepackage{color}
\usepackage{tikz}
\usepackage{enumerate}
\usepackage[english]{babel}
\usepackage{xcolor}
\usepackage{comment}
\newtheorem{thm}{Theorem}[section]
\newtheorem{cor}[thm]{Corollary}
\newtheorem{lem}[thm]{Lemma}
\newtheorem{prop}[thm]{Proposition}

\theoremstyle{remark}
\newtheorem{rem}[thm]{Remark}
\theoremstyle{definition}
\newtheorem{defn}[thm]{Definition}
\makeatletter
\renewcommand*\env@matrix[1][\arraystretch]{%
  \edef\arraystretch{#1}%
  \hskip -\arraycolsep
  \let\@ifnextchar\new@ifnextchar
  \array{*\c@MaxMatrixCols c}}
  \makeatother
\DeclareMathOperator{\sgn}{sgn}

\def \rm{\mathbb{R}}
\def \rgg{\mathbf{g}}
\def \d{\partial}
\def \ds{\d_{s}}
\def \dss{\d_{ss}}
\def \dt{\d_{t}}
\def \e{\eta^1}
\def \ee{\eta^2}
\def \eee{\eta^3}
\def \ei{\eta^{i}}
\def \ii{\int_{0}^{1}}
\def \kki{\psi^{i}}
\def \k{\psi}
\def \si{\sigma^{i}}

\def \su{\sum_{i=1}^3}
\def \rm{\mathbb{R}}
\def \ep{\epsilon}
\def \epp{_{\ep}}
\def \eiee{\eta_{\epsilon}^{i}}

\def \io{\int_{\Omega}}
\def \it{\int_{\mathfrak Q_{T}}}

\def \ki{\psi^{i}}

\def \itt{\int_{\mathfrak Q_{\mathfrak t}}}

\def \l{\left(}
\def \r{\right)}
\def \iii{^{i}}
\def \tt{\mathfrak t}
\def \ai{\alpha^{i}\l s\r}
\date{}

\title[Falling Network]{Overdamped dynamics of a falling inextensible network: existence of solutions}

\author[A.~Telciyan]{Ayk Telciyan}
\address[A.~Telciyan]{University of Coimbra, CMUC, Department of Mathematics,  3001-501 Coimbra, Portugal}{}
\email{ayktelciyan@mat.uc.pt}
\author[D.~Vorotnikov]{Dmitry Vorotnikov}
\address[D.~Vorotnikov]{University of Coimbra, CMUC, Department of Mathematics,  3001-501 Coimbra, Portugal}{}
\email{mitvorot@mat.uc.pt}

\begin{document}

	\begin{abstract}  We study the equations of overdamped motion of an inextensible triod with three fixed ends and a free junction under the action of gravity. The problem can be expressed as a system of PDE that involves unknown Lagrange multipliers and non-standard boundary conditions related to the freely moving junction. 
		It can also be formally interpreted as a gradient flow of the potential energy on a certain submanifold of the Otto-Wasserstein space of probability measures. We prove global existence of generalized solutions to this problem.
	\end{abstract}
	\maketitle
	
	\vspace{10pt}
	
	\section{Introduction} \label{s:in}
	
An \emph{inextensible network} is a union of several inextensible strings that meet at some of their endpoints called \emph{junctions}. The study of inextensible networks from the mathematical perspective was started a long time ago by Chebyshev \cite{Tcheb} and Rivlin \cite{rivlin}, aiming at modelling textile fabrics. Aside from \cite{Nov20}, we are however not aware of any investigation of evolutionary behavior of inextensible networks. Our paper thus seems to be one of the first contributions to this particular field. On the other hand, there has been a major recent activity on well-posedness of geometric flows describing time-evolving extensible networks, see \cite{MN04,MN16,Gar19,Gar20,Aq19,Aq21,KrNov21} and the survey \cite{man21}. Whereas the authors of \cite{MN04,MN16,KrNov21} deal with variants of the mean curvature flow for networks, \cite{Nov20} and the other mentioned articles consider \emph{elastic flows} (interpolations between the mean curvature flow and the Willmore flow). 

The main technical difficulties that appear in the study of networks in contrast with the evolution of single strings are due to the rather non-standard boundary conditions at the junction points. Accordingly, to fix the ideas, we decided to restrict ourselves to the simplest possible network with only one junction, the so-called \emph{triod}, cf. \cite{Aq19,MN04}. Our triod consists of three inextensible strings (\emph{arms} of the triod) that meet at a common point (junction), and the remaining ends are fixed at three distinct points of $\mathbb R^d$, $d\in \mathbb N$, $d>1$ (the physically relevant cases are $d=3$ and to a lesser extent $d=2$). Note that the junction is moving in an unknown way and thus constitutes a kind of a \emph{free boundary}. The evolving triod is allowed to self-intersect, and the subtle issue whether the embeddedness of the network is preserved (this might be particularly challenging for $d=2$) lies beyond the scope of the paper.  

The literature on flows of networks cited above is concerned with variational evolution driven by ``intrinsic'' energies (related to the length or curvature). In this paper we investigate the gradient flow of an ``extrinsic'' energy, namely, the potential energy determined by an external force (gravity), with respect to a suitable geometry, cf. \cite{shidim1}. As we explain below, this models the overdamped motion of a falling inextensible network (triod). 

It is impossible to discuss evolution of inextensible networks without referring to the state of the art for single inextensible cords (we will not touch upon the extensible cords because the amount of the corresponding literature is enormous). Various elastic flows of inextensible strings were studied in \cite{wen93,Koi96,Oel11,Oel14,Ok07,O08,Schwet,OW20}. The presence of elastic forces contributes towards non-degenerate parabolicity of the flows and helps to overcome the difficulties caused by the Lagrange multipliers related to the inextensibility constraint; in our situation such forces and hence such an advantage are missing.  Our paper has particularly been influenced by \cite{shidim1} (that studied the overdamped dynamics of a falling whip) and \cite{shidim2} (that dealt with the ``uniformly compressing''\footnote{The uniformly compressing mean curvature flow is, roughly speaking,  the closest flow to the classical mean curvature flow among the flows that uniformly contract the volume of the evolving manifold, see \cite{shidim2} for the rigorous definition. In the case of evolving loops, after a suitable change of variables it becomes quite similar to the overdamped flow of \cite{shidim1}.} counterpart of the mean curvature flow).  

The full dynamical equations (Hamiltonian systems) governing the motion of inextensible strin\-gs (with or without elastic forces) are very tricky. The literature about the solvability of the corresponding initial-boundary value problems is scarce and includes the studies near the equilibrium \cite{Reeken-I}, locally in time \cite{Preston-2011}, and globally in time in various severely relaxed senses \cite{grothaus,JDE17}.

The equations of motion of an inextensible triod in the ambient space $\mathbb R^d$ subject to gravity force can be derived from the least action principle by following the road map from \cite[Section 2.6]{JDE17}. Assuming for simplicity that the length of each arm of the triod is equal to $1$, the  initial-boundary value problem reads

	\begin{equation}\label{triodeq}
	\begin{cases}
		\d_{tt}\eta^{i}=\d_{s}\left(\sigma^{i}\d_{s}\eta^{i}\right)+g,\\
		|\ds\eta^i|=1,\
	\end{cases} 
\end{equation} 
subject to 
the boundary conditions
\begin{equation}\label{boundary}
	\begin{cases}
		\e\l t,0\r=\ee\l t,0\r=\eee\l t,0\r,\\
		\eta^{i}\l t,1\r=\alpha\iii\l 1\r,\\
		\sigma^1\ds\e+\sigma^2\ds\ee+\sigma^3\ds\eee=0~\text{at}~ s=0 ~\text{for all}~t,
	\end{cases}
\end{equation}
and the initial conditions 
\begin{equation}\label{initial}
	\ei\l 0,s\r=\alpha\iii\l s\r,
\end{equation}
\begin{equation}\label{initial2}
	\partial_t\ei\l 0,s\r=\beta\iii\l s\r.
\end{equation}
Here $\eta^{i}=\eta^{i}(t,s)\in \mathbb R^d$, $i=1,2,3$, is the position vector at time $t\geq 0$ of the particle that is labelled by the arc length parameter $s$ and belongs the $i$th arm of the triod. For each $i$, the scalar function $\sigma^i=\sigma^i(t,s)$ is the Lagrange multiplier (that is often referred to as the \emph{tension}) coming from the inextensibility of the $i$th arm. Finally, $g$ is a constant gravity vector for which we assume w.l.o.g. that $|g|=1$, and $\alpha^i(s), \beta^i(s)$ determine the initial dynamical configuration of the triod. Note that $s=1$ corresponds to the fixed ends, and $s=0$ corresponds to the (moving) junction. 

From the geometrical point of view, a natural infinite-dimensional configuration manifold for the evolving triods is $$\mathcal{A}=\{\eta=(\eta^1,\eta^2,\eta^3):\ei\in H^2\left( 0,1;\rm^{d} \right),\ \e\l 0\r=\ee\l 0\r=\eee\l 0\r,\ \eta^{i}\l 1\r=\alpha\iii \l 1\r,\ |\d_{s}\ei(s)|=1\, \forall s\in[0,1] \}$$ viewed as a submanifold of $L^2(0,1; \rm^{3d})$ (and hence
 equipped with a weak Riemannian metric). Observe that the tangent space at a ``point'' $\eta$ is 
 \begin{equation}\label{tangentspace} T_\eta\mathcal{A}=\{v=(v^1,v^2,v^3):v^i\in H^2\left( 0,1;\rm^{d} \right),\ v^1 \l 0\r=v^2 \l 0\r=v^3 \l 0\r,\ v^{i}\l 1\r=0,\ \ds\ei(s)\cdot \ds v^i(s)=0 \}. \end{equation}
  Here and below the dot stands for the Euclidean product in $\mathbb R^d$. Note that we never employ Einstein's summation convention. Then \eqref{triodeq}, \eqref{boundary} is at least formally equivalent to  Newton's equation \begin{equation}\nabla_{\dot\eta}\dot\eta=-\nabla_{\mathcal{A}}E\left(\eta\right). \label{newton}
\end{equation}
Here \begin{equation}
E(\eta):=\su\ii-g\cdot\ei(s)\,ds
\end{equation} is the potential energy of a triod. 
 
 The Riemannian manifold $\mathcal A$ (as well as its counterparts for single strings, cf. \cite{Preston-2012,shidim1,shidim2}) has some interesting features. It can be viewed, cf. \cite{shidim2},  as a submanifold of the Otto-Wasserstein space of probability measures \cite{otto01,villani03topics,villani08oldnew} from the optimal transport theory (this in particular implies that the geodesic distance on $\mathcal A$ does not vanish, being bounded from below by the Wasserstein distance, which is in stark contrast with the underlying geometry of the mean curvature, Willmore and similar flows, cf. \cite{MM5, MM06,MM07,BHM, BBM14}). It can also be regarded as a particular case ($m=1$) of the manifolds of $m$-dimensional \emph{incompressible membranes} (in other words, of volume preserving immersions), cf. \cite{BMM16,Mol17}. The opposite borderline case $m=d$ tallies with Arnold's formalism \cite{Ar66,Ar98} for ideal incompressible fluids or rather, even more specifically, with  the motion of fluid patches in $\mathbb R^d$, which has recently been studied \cite{Sle19} from a similar perspective. However, in Arnold's case ($m=d$) the manifold has a Lie group structure, which allows one to work in the corresponding Lie algebra (i.e., in the mechanical language, to use the Eulerian coordinates). In our case $m=1$, there is no Lie algebra structure, and the Lagrangian description as in \eqref{triodeq}, \eqref{boundary}, \eqref{initial},  \eqref{initial2} seems to be unavoidable. 
 
 If the fall of the triod is overdamped by a heavily dense environment, the equations of motion \eqref{triodeq} become \begin{equation}\label{nwgrd}
 	\begin{cases}
 		\d_{t}\eta^{i}=\d_{s}\left(\sigma^{i}\d_{s}\eta^{i}\right)+g,\\
 		|\ds\eta^i|=1.\
 	\end{cases} 
 \end{equation} 
Note that the velocity $\d_{s}\left(\sigma^{i}\d_{s}\eta^{i}\right)+g$ acts not only in the direction normal to the curves constituting the network, but there is a tangential motion as well. 

We refer to \cite{shidim1} for the details of the derivation of \eqref{nwgrd} in the case of a single cord.  It is also possible to directly obtain the overdamped flow \eqref{nwgrd} from the full dynamical equation \eqref{triodeq}
by employing the quadratic change of time, cf. \cite{BD18}. Finally, our problem \eqref{nwgrd}, \eqref{boundary} can be realized as the gradient flow  
of the potential energy $E$ on the manifold $\mathcal A$, i.e.,  \begin{equation}\label{gradequality}
\dot\eta=-\nabla_{\mathcal{A}}E\left(\eta\right). 
\end{equation}

In light of the previous discussion (see also \cite{Preston-2011,Preston-2012,JDE17,thess}) equation \eqref{triodeq} has much in common with the Euler equation of ideal incompressible fluid. In the same spirit, the overdamped equation \eqref{nwgrd} is comparable to the Muskat problem (also known as the incompressible porous medium equation) that received a lot of attention during the last decade, see \cite{annmuskat,szemuskat,Muskatajm,Muskatinv} and the references therein.  

In this article, we are interested in constructing global in time solutions to \eqref{nwgrd}, \eqref{boundary}, \eqref{initial}. We deal with generalized solutions, which allows us to consider not necessarily smooth but merely rectifiable triods.

In what follows, we denote $\Omega:=(0,1)$, $\mathfrak Q_{t}:=(0,t)\times\Omega$ for $t\in(0,\infty]$ and $\rgg(s):=(g,g,g)\in L^2(\Omega;\mathbb{R}^{3d})$, $d\in \mathbb N$, $d>1$.

\begin{rem}[Initial data]\label{remofradius} We fix once and for all Lipschitz initial data $\alpha^{i}\in W^{1,\infty}\left(\Omega\right)^d$, $i=1,2,3$, satisfying the compatibility conditions \begin{equation}\label{compsimple}
		\alpha^{1}\l 0\r=\alpha^{2}\l 0\r=\alpha^{3}\l 0\r=0
	\end{equation}and \begin{equation}\label{alphan}|\ds\alpha^{i}\l s\r|=1\ \textrm{a.e.}\ \textrm{in}\ \Omega. \end{equation}  Since \eqref{alphan} is only required to hold almost everywhere, the arms of the triod can have shape of any rectifiable curve at the initial moment.  Note that we have also w.l.o.g. assumed that the junction is located at the origin at the initial moment. 
	We will moreover assume that the arms of the triod are not fully straight at the initial moment which means that $|\alpha^{i}\l 1\r|<1$ (since the length of each arm is equal to $1$), $i=1,2,3$. 
\end{rem}

Our goal is to prove the following main result.  
\begin{thm}[Global existence of generalized solutions]\label{thm17}
	For every initial configuration $\alpha^{i}\l s\r\in W^{1,\infty}\l\Omega\r^d$, $i=1,2,3$, meeting the assumptions of Remark \ref{remofradius}, there exists a generalized solution to \eqref{nwgrd}, \eqref{boundary},  \eqref{initial} in $\mathfrak{Q}_{\infty}$. Moreover, those solutions satisfy $\sigma\iii\l t,s\r\geq 0 $ for almost every $\l t,s\r\in \mathfrak{Q}_{\infty}$.
\end{thm}
Note that the precise definition of a generalized solution is lengthy and will be introduced in Definition \ref{def1}. 

Observe that $\mathcal{A}$, being a formal submanifold of the Otto-Wasserstein space, is a metric space with a non-degenerate (Riemannian) distance. Nevertheless, $\mathcal{A}$ is neither a complete metric space nor a geodesic space. Accordingly, the theory of gradient flows in metric spaces, cf. \cite{AGS08,villani08oldnew}, does not sound to be applicable to well-posedness of our flow \eqref{gradequality}.

To achieve our goal, we will follow the strategy suggested by Shi and the second author \cite{shidim1} for the evolution of a single string. It basically consists in approximation of the original gradient flow on $\mathcal A$ by suitable gradient flows on the ambient space $L^2(\Omega; \rm^{3d})$. The idea is to derive uniform  estimates for the approximating problem that would allow us to pass to the limit and to show that the limiting functions are solutions to \eqref{nwgrd}, \eqref{boundary},  \eqref{initial}. However, because of the complicated boundary conditions \eqref{boundary}, many of the  estimates that were used in \cite[Section 3]{shidim1} fail to be generalizable to our setting. This in particular applies to the crucial $L^\infty$ estimate \cite[Proposition 3.3]{shidim1} in the spirit of Ladyzhenskaya, Solonnikov and Uraltseva, cf. \cite{asu}. We will manage to overcome these difficulties and to prove novel and more refined estimates by leveraging the gradient flow structure of the approximating problem much more thoroughly than in \cite{shidim1}. This will be combined with careful observations involving geometric properties of triods, the behaviour of the curvature and some convexity argument (cf. Lemmas \ref{lemmatre} and \ref{lemmakappa} below).

Apart from that, in \cite{shidim1} the existence of $C^\infty$-smooth solutions to the approximating problem was immediate from Amann's theory, cf. \cite{Amann}. It is not applicable here anymore (again due to the boundary conditions), so we will solve our approximate problem (see our Corollary \ref{cor:sol}) by the theory of abstract evolution equations with pseudomonotone maps, cf. \cite{roubicek}.

The paper is organized as follows. In Section \ref{s:app}, we heuristically motivate and then introduce the approximating problem. In Section \ref{evolsec}, we prove its solvability. The main technical work is done in Section \ref{s:sec4} where we establish various uniform estimates for the approximating problem. A highlight of that section is the crucial and ingenious $L^\infty$ bound for the tension (Lemma \ref{lemmakappa}). In Section \ref{s:sec5}, based on the results of Section \ref{s:sec4}, we will be able to pass to the limit and to prove Theorem \ref{thm17}. Our results still hold for the overdamped dynamics of a falling single cord with two fixed ends, see Remark \ref{remst} and Proposition \ref{corst}. In Appendix we provide a formal computation of the gradient of the potential energy $\nabla_{\mathcal A} E$. 
	
	\section{Approximating problem} \label{s:app}

Let us now describe the way of approximation of our gradient flow that we plan to employ in order to prove Theorem \ref{thm17}. Here we follow the road map of \cite{shidim1}.
	
We begin with some heuristics. Consider the extra variables $\kki=\si\ds\ei$, $i=1,2,3$. Then our system \eqref{nwgrd} can at least formally be rewritten as 
	\begin{equation}\label{nts}
		\begin{dcases}
			\dt\ei=\ds\kki+g,\\
			\kki=\si\ds\ei,\\
			\si=\kki\cdot\ds\ei.
		\end{dcases}
	\end{equation}
 More precisely, the constraints $|\ds\ei|=1$ yield $|\kki|=|\si|$ and  $\kki=\sgn(\si)|\kki|\ds\ei$. We make the ansatz $\si\geq0$ (that will be a posteriori justified by Theorem \ref{thm17}) and infer $\kki\cdot\ds\ei=\si$ (see also Remark \ref{nonconv} below for a related discussion). Note that we  formally have $\ds\ei=\frac{\kki}{|\kki|}$, thus the map $\kki\mapsto \ds\ei$ is not a diffeomorphism. To overcome this issue, 
 we fix $\ep\in(0,1)$ and introduce the auxiliary functions \begin{equation}\label{dff}
 	F\epp:\rm^d\rightarrow\rm^d, ~ F\epp(\k):=\ep\k+\frac{\k}{\sqrt{\ep+|\k|^2}}
 \end{equation}
 and  $$G_{\epsilon}\left(\tau\right):=\left(F\epp\right)^{-1}(\tau).$$
 Approximating the relations $\kki\mapsto \ds\ei$ and $\ds\ei\mapsto \kki$ by $F\epp$ and $G\epp$, respectively, leads from \eqref{nwgrd}, \eqref{boundary},  \eqref{initial} to the problem
	\begin{equation}\label{appproblem}
		\dt\eiee=\ds\left(G\epp\left(\ds\eiee\right)\right)+g,\quad i=1,2,3,
	\end{equation} 
	with the following initial and boundary conditions:
	\begin{equation}\label{appib}
		\begin{aligned}
			\eiee\left(0,s\right)=\alpha^{i}\l s\r&,\\
			\eiee\l t,1\r=\alpha\iii\l 1\r&,\\
			\eta_{\epsilon}^{1}\l t,0\r=\eta_{\epsilon}^{2}\l t,0\r=\eta_{\epsilon}^{3}\l t,0\r&,\\
			\su G\epp\left(\ds\eiee\right)=0~\text{at}~s=0~\text{for all}~t&.
		\end{aligned}
	\end{equation}
	\begin{rem}\label{rem0}
		Let us make an elementary observation that is very important in the sequel. The Euclidean norm  $|F\epp\l\psi\r|$ depends only on $|\psi|$ and is an increasing function of $|\psi|$. If $|\psi|= 1$, then by simple calculation $|F\epp\l\psi\r|>1$. Consequently, if $|\tau|\leq 1$, then $|G\epp(\tau)|< 1$.
	\end{rem}
	
	By explicit computation, $\nabla G\epp$ is positive-definite and 
	\begin{equation}\label{lambdainequality}
		\lambda\epp\left(\tau\right)|\xi|^2\leq\nabla G\epp\left(\tau\right)\xi\cdot\xi\leq \Lambda\epp\l\tau\r|\xi|^2,~~\forall\xi\in\rm^d,\tau\in\rm^d,
	\end{equation}
	where $\Lambda\epp$ and $\lambda\epp$ satisfy
	\begin{equation}\label{lambdaineq}
		\begin{aligned}
			\frac{1}{\ep+\ep^{-1/2}}\leq \lambda\epp\left(\tau\right)=\frac{1}{\ep+\left(\ep+|G\epp(\tau)|^2\right)^{-1/2}},\\
			\Lambda\epp\left(\tau\right)=\frac{\ep^{-1}}{1+\left(\ep+|G\epp(\tau)|^2\right)^{-3/2}}\leq  \ep^{-1}.
		\end{aligned}
	\end{equation}
	
	Motivated by the original system \eqref{nts}, given a solution $\eta\epp$ to the approximating problem \eqref{appproblem}, \eqref{appib} we define
	\begin{equation}\label{defeq}
		\kki\epp:=G\epp\left(\ds\ei\epp\right),~~\si\epp:=G\epp\left(\ds\ei_{\ep}\right)\cdot\ds\ei\epp.
	\end{equation}
	
	Observe from the definition of $G\epp$ that there exists a bounded smooth positive scalar function $\gamma\epp$ such that $G\epp(\tau)=\gamma\epp(|\tau^2|)\tau$, $\tau\in \rm^d$. In particular, this implies that 
	\begin{equation}\label{sinon}
		\si\epp\geq 0.
	\end{equation}  Moreover, $\gamma\epp$ is bounded away from $0$ and $\infty$ (not uniformly w.r.t. $\ep$).  Let $\Gamma\epp$ be the primitive of $\gamma\epp$ with $\Gamma\epp(0)=0$.  Set $$Q\epp(\tau):=\frac 1 2  \Gamma\epp(|\tau|^2).$$ Observe that \begin{equation} \label{compgrad} \nabla Q\epp(\tau)=G\epp(\tau).\end{equation} Moreover, $Q\epp$ can be computed explicitly: \begin{equation}\label{gdef}
		Q\epp\left(\tau\right)=\epsilon\left(\frac{|G\epp(\tau)|^2}{2}-\frac{1}{\sqrt{\ep+|G\epp(\tau)|^2}}\right)+\sqrt \ep.
	\end{equation} By Remark \ref{rem0}, $Q(\tau)<<1$ if $|\tau|\leq 1$. 

	We define the associated ``total energy'' of the approximating problem \eqref{appproblem}, \eqref{appib} by
	\begin{equation}\label{approxenergy}
		\bm{\mathcal{E}}\epp\left(\eta\right):=\begin{dcases}
			\su\left(\ii Q\epp\left(\ds\ei\right)ds+\ii\left(-g\right)\cdot\ei ds\right)\\ \textrm{for}\   \eta\in AC^2(\Omega; \rm^{3d})\ \textrm{satisfying}\ \eta(1)=\alpha(1),\ \eta^{1}\l 0\r=\eta^{2}\l 0\r=\eta^{3}\l 0\r; \\ 
			+\infty \ \textrm{for}\ \textrm{any}\ \eta\in L^2(\Omega; \rm^{3d})\ \textrm{except}\ \textrm{those}\ \textrm{above}. 
		\end{dcases}
	\end{equation}

	Then \eqref{appproblem}, \eqref{appib} can at least formally be interpreted as a gradient flow, with respect to the flat Hilbertian structure of $L^2(\Omega; \rm^{3d})$, that is driven by this functional, i.e.$$\dot\eta=-\nabla_{L^2(\Omega; \rm^{3d})}\bm{\mathcal{E}}\epp\left(\eta\right), \quad \eta(0)=\alpha.$$ We will return to this issue in the next section.

\section{Evolution by pseudomonotone maps and solvability of the approximating problem} \label{evolsec}
For the existence of the solution to the approximating problem,  we use the theory of abstract evolution equations involving pseudomonotone maps. We prefer this approach  (instead of directly employing the theory of gradient flows in Hilbert spaces, cf. \cite{brezisbook,butatm}) because it automatically gives us the regularity of solution that is required for  the manipulations of Section \ref{s:sec4}. 

Let us start with introducing some concepts and definitions, mainly following the book \cite{roubicek}.  
Let $V$ be a separable reflexive Banach space, and $V^*$ be the dual space of $V$. We use the bracket notation for the duality. Assume that there is a
continuous embedding operator $i:V\to H$, and $i(V)$ is dense in
$H$, where $H$ is a Hilbert space. This generates the \emph{Gelfand triple} $V\subset H\subset V^{*}$ by the following well-known observation. The adjoint operator $i^*:H^*\to V^*$ is continuous and,
since $i(V)$ is dense in $H$, one-to-one. Since $i$ is one-to-one,
$i^*(H^*)$ is dense in $V^*$, and one may identify $H^*$ with a
dense subspace of $V^*$. Due to the Riesz representation theorem,
one may also identify $H$ with $H^*$.  Moreover, the $H$-scalar product of  $f\in H, u\in V$
coincides with the value of the functional $f$ from $V^*$ on the
element $u\in V$:
\begin{equation} \label{ttt} (f,u)_H=\langle f,u \rangle. \end{equation} \begin{defn}
	A mapping $A: V\rightarrow V^*$ is called monotone if  $ \forall u,v\in V$  we have $\langle A\l u\r-A\l v\r,u-v\rangle\geq 0 .$
\end{defn}
\begin{defn} A mapping $A: V\rightarrow V^*$ is called radially continuous if $ \forall u,v\in V: t\mapsto \langle A\l u+vt\r,v\rangle$ is continuous. 
\end{defn}

\begin{defn}
	A mapping $A: V\rightarrow V^*$ is called pseudomonotone provided\\
	 (i) $A$ is bounded (i.e., the image of any bounded set is bounded),
	 \\
	 (ii) for any sequence $u_{k}\rightharpoonup u$ weakly with $$\limsup_{k\rightarrow\infty}\langle A\l u_{k} \r,u_{k}-u \rangle\leq 0$$ and for every $v\in V$ it is true that  $$\langle A\l u\r,u-v \rangle\leq \liminf_{k\rightarrow\infty}\langle A\l u_{k}\r, u_{k}-v \rangle.$$
\end{defn} 
We will need the following useful criterion of pseudomonotonicity from \cite{brezis}.
\begin{lem}\label{brezis} A bounded,  radially continuous and monotone mapping  is pseudomonotone.
\end{lem}

Assume that there is a seminorm $|\cdot|_V$ on $V$ that satisfies the ``abstract Poncar\'e inequality'' $$\|u\|_V\lesssim \|u\|_H+|u|_V, \quad \forall u\in V,$$ where $\|\cdot\|_{H}$  is the Eucledian norm in $H$.

\begin{defn} A mapping $A: V\rightarrow V^*$ is called semicoercive if for $u\in V$ we have $$\langle A\l u\r, u\rangle\geq c_{0}|u|^2_{V}-c_{1}|u|_{V}-c_{2}\|u\|^2_{H},$$ where $c_{0},c_{1}$ and $c_{2}$ are nonnegative constants.
\end{defn}

 Consider the following abstract initial value problem on the time interval $(0,T)$: \begin{equation}\label{bsystem}\frac{d}{dt}u+A\l u\l t\r\r=f(t),\ u\l 0\r=u_{0}.
\end{equation} The following result can be found in \cite[Theorem 8.18]{roubicek}. 
\begin{thm}
	\label{rubicekthm} Let $A:V\rightarrow V^*$ be a pseudomonotone and semicoercive mapping and \begin{itemize}
		\item[] $f\in AC^2\l [0,T],V^{*}\r$,
		\item[] $u_{0}\in V$ is such that  $A\l u_{0} \r-f(0)\in H$,
		\item[] $\langle A\l u_{1}\r -A\l u_{2}\r, u_{1}-u_{2} \rangle\geq c_{0}|u_{1}-u_{2}|^2_{V}-c_{2}\|u_{1}-u_{2}\|^2_{H}$ for $u_1, u_2\in V$ with some constants $c_{0}, c_2>0$. 
	\end{itemize}
	Then there exists $u\in W^{1,\infty}\l 0,T;H \r\cap AC^2\l [0,T]; V\r$ that solves the Cauchy problem \eqref{bsystem} (the first equality in \eqref{bsystem} holds in the space $V^*$ for a.a. in $(0,T)$, whereas the second one holds in the space $V$).
\end{thm}

Our next goal is to apply this theorem to show the existence and regularity of solutions to the approximating problem. It will be convenient to rewrite our approximating problem \eqref{appproblem}-\eqref{appib} with the help of the transformation $$\xi^{i} \l t,s \r:=\eta^{i}\epp\l t,s\r-\alpha^{i}\l s\r,$$ i.e., we simply subtract the initial data, arriving at
\begin{equation}\label{nwsystemwithxi}
	\begin{dcases}
		\dt \xi^{i}-\ds\left( G\epp\left(\ds\l\xi^{i}+\alpha^{i}\right)\r\r=g, \quad i=1,2,3,\\
		\xi^{1}\l t,0\r=\xi^{2}\l t,0\r=\xi^{3}\l t,0\r,\\
		\xi^{i}\l t,1\r=0,\\
		\xi^{i}\l 0,s\r=0,\\
		\su G\epp\left(\ds\l\xi^{i}+\alpha^{i}\r \right)(t,0)=0.
	\end{dcases}
\end{equation}
  In order to recast this system in the form of the Cauchy problem \eqref{bsystem}, we let $$H=L^2 \l \Omega;\rm^{3d}\r$$ be the Hilbert space of triples with the natural scalar product. Let also  $$V:=\{u=\{u^{i}\}\in  AC^2\l \overline\Omega;\rm^{3d}\r \text{ such that } u^{i}(1)=0 \text{ and } u^{1}\l 0\r=u^{2}\l 0\r=u^{3}\l 0\r\}$$ and $V^*$ be the corresponding dual space. 
   
  It is a separable reflexive Banach space with the norm inherited from $H^1$.    Define a seminorm on $V$ by $|\{u^i\}|_V:=\|\{\ds u^i\}\|_H$. The required Poincar\'e inequality obviously holds.  Let $\bm A: V\rightarrow V^*$ be the mapping that is defined by duality as follows: \begin{equation}\label{d:a}\langle\bm{A}\l\xi\r,\zeta \rangle=\su\int_{0}^{1}G\epp\l\ds\l\xi^{i}+\alpha^{i} \r\r\cdot\ds\zeta^i ds.\end{equation} 
Then \eqref{nwsystemwithxi} rewrites as 
\begin{equation}\label{nwsystemwithxiab}\frac{d}{dt}\xi+\bm A\l \xi\l t\r\r=\rgg,\ \xi\l 0\r=0.
\end{equation} Note that the last equality of \eqref{nwsystemwithxi} is hidden in the duality in \eqref{d:a}. 

In order to check that Theorem ~\ref{rubicekthm} is applicable to \eqref{nwsystemwithxiab} we need to prove several auxiliary statements.  For the sake of readability, we will omit the subscript $\ep$ coming from the approximating problem.  

\begin{lem}\label{mainineqaulity}
	The mapping $\bm{A}$ satisfies the inequality $$\langle \bm A\l \xi_{1}\r -\bm A\l \xi_{2}\r, \xi_{1}-\xi_{2} \rangle\geq c_0|\xi_{1}-\xi_{2}|^2_{V}$$ 
	for some constant $c_{0}>0$ (depending on $\epsilon$) and any $\xi_1, \xi_2 \in V$. 
\end{lem}
\begin{proof} Define $A^i: H^1\l \Omega;\rm^d\r\to \l H^1\l \Omega;\rm^d\r\r^*$ by \begin{equation*}\langle A^i\l\xi^i\r,\zeta \rangle=\int_{0}^{1}G\epp\l\ds\l\xi^{i}+\alpha^{i} \r\r\cdot\ds\zeta^i ds.\end{equation*} 
	Throughout the rest of the proof, we omit the index $i$ to avoid heavy notation in $A^i$, $\xi_1^i$, $\xi_2^i$ and $\alpha^{i}$. 
	With this convention, it suffices to prove that $$\langle A\l \xi_{1}\r -A\l \xi_{2}\r, \xi_{1}-\xi_{2} \rangle\geq c_0\|\ds(\xi_{1}-\xi_{2})\|^2_{L^2}.$$ We compute
	\begin{align} \label{computa}
		\langle A\l \xi_{1}\r- A\l \xi_{2}\r,\xi_{1}-\xi_{2}\rangle=\int_{\Omega} \left[G\l \ds\l \xi_{1}+\alpha\r\r-G\l \ds\l \xi_{2}+\alpha\r\r\right]\cdot\ds\l \l \xi_{1}+\alpha\r -\l \xi_{2}+\alpha \r \r ds.
	\end{align}
	Let us denote $\mu:=G\l \ds\l \xi_{1}+\alpha\r\r$ and $\gamma:=G\l \ds\l \xi_{2}+\alpha\r\r$. Now, we use the relation between $F$ and $G$ and conclude that $F\l \mu \r=\ds\l \xi_{1}+\alpha\r$ and $F\l \gamma \r=\ds\l \xi_{2}+\alpha\r$.  We can rewrite the right-hand side of \eqref{computa} as \begin{align*}
		&\int_{\Omega}( \mu-\gamma)\cdot \l F\l\mu\r-F\l\gamma \r\r ds\\ =&\int_{\Omega} \l \mu-\gamma \r \cdot \l \ep\l\mu-\gamma\r+\frac{\mu}{\sqrt{\ep+|\mu|^2}}-\frac{\gamma}{\sqrt{\ep+|\gamma|^2}} \r ds\\ \geq &\int_{\Omega }\ep|\mu-\gamma|^2 ds
	\end{align*} because the map $r\mapsto \frac{r}{\sqrt{\ep+r^2}}$ is a gradient of a convex function. 
	Observe that $$|F\l\mu\r-F\l \gamma\r|\leq (\ep+\ep^{-1/2})|\mu-\gamma|$$ by the mean value theorem and the Cauchy-Schwarz inequality since the operator norm of the matrix $\nabla F(r)$ is bounded from above by $\ep+\ep^{-1/2}$, cf. \eqref{lambdaineq}.
	
	Thus we conclude that $$\langle A\l \xi_{1}\r- A\l \xi_{2}\r,\xi_{1}-\xi_{2}\rangle\geq \ep\int_{\Omega}|\mu-\gamma|^2 ds\geq c_0 \int_{\Omega}|F(\mu)-F(\gamma)|^2 ds=c_0\|\ds(\xi_{1}-\xi_{2})\|^2_{L^2}.$$
	
\end{proof}
\begin{cor} 
	The mapping $\bm{A}$ is monotone.
\end{cor}
\begin{proof}
	It is clear from Lemma \ref{mainineqaulity}.
\end{proof}
\begin{cor} 
	The mapping $\bm{A}$ is semicoercive.
\end{cor}
\begin{proof}
	Employing Lemma \ref{mainineqaulity} and Cauchy-Schwarz inequality, we see that
	\begin{align*}
		\langle \bm A\l \xi\r, \xi\rangle&=\langle \bm A\l \xi\r-\bm A\l 0\r,\xi\rangle+\langle \bm A\l 0\r,\xi \rangle\\
		&\geq|\xi|^2_{V}+\langle \bm A\l 0\r,\xi\rangle\\
		&=|\xi|^2_{V}+\su\int_{0}^{1}G\l\ds\alpha^{i} \r\cdot\ds\xi^i ds\\
		&\geq |\xi|^2_{V}-\left\|\left\{G\l\ds\alpha^i \r\right\}\right\|_H|\xi|_{V}\\
		&\geq |\xi|^2_{V}-c_2|\xi|_{V},
	\end{align*}
	where $c_2$ is a positive constant depending on $\alpha$.
\end{proof}
\begin{lem}
	The mapping $\bm{A}$ is bounded.
\end{lem}
\begin{proof} Indeed,
	\begin{align*}
		\langle \bm A\l\xi\r,\zeta \rangle&=\su\int_{0}^{1}G\l\ds\l\xi^{i}+\alpha^{i} \r\r\cdot\ds\zeta^i ds\\
		&\leq \left\|\left\{G\l\ds\l\xi^i+\alpha^i \r\r\right\}\right\|_{H}|\zeta|_{V}\\
		&\lesssim \left\|\left\{\ds\l\xi^i+\alpha^i \r\right\}\right\|_{H}|\zeta|_{V}\\
		&\leq|\xi+\alpha|_{V}\|\zeta\|_{V}.
	\end{align*}
(We have used sublinearity of $G$). Since $|\alpha|_V$ is finite, this implies that $\|\bm A(\xi)\|_{V^*}$ is bounded provided $\|\xi\|_V$ is bounded. 
\end{proof}
\begin{lem}
	The mapping $\bm{A}$ is radially continuous.
\end{lem}
\begin{proof}
	Fix $\xi,\zeta\in V$ and let $\tau_{n}\rightarrow\tau$ be a sequence. Then it is easy to see that $$\su  G\l\ds\l\xi^{i}+\tau_n\zeta^i+\alpha^{i} \r\r(x)\cdot\ds\zeta^i(x)\to \su G\l\ds\l\xi^{i}+\tau\zeta^i+\alpha^{i} \r\r(x)\cdot\ds\zeta^i(x)$$ a.e. in $\Omega$.  The claim will follow from Lebesgue's dominated convergence theorem if there is a  function in $L^1(\Omega)$ that dominates the left-hand side. But it is indeed the case since we can leverage sublinearity of $G$ to estimate $$\left|\su  G\l\ds\l\xi^{i}+\tau_n\zeta^i+\alpha^{i} \r\r\cdot\ds\zeta^i\right|\\ \leq C |\ds\l\xi+\tau_n\zeta+\alpha\r| \cdot |\ds\zeta|\leq C ( |\ds\xi|^2+|\ds\zeta|^2+ |\ds\alpha|^2),$$ and the right-hand side is $L^1$ by the assumption.
\end{proof}
We can now legitimately use Theorem~\ref{rubicekthm} in order to solve \eqref{nwsystemwithxiab}. 
\begin{cor}
	Given $\alpha$ as in Remark \ref{remofradius}, the system \eqref{nwsystemwithxiab} has a solution $\xi=\{\xi^{i}\}\in  W^{1,\infty}\l 0,T;H \r\cap AC^2\l [0,T]; V\r$ that is understood in the same sense as in Theorem \ref{rubicekthm}.
\end{cor}
Returning back to the variable $\eta$ and leveraging elementary properties of $G\epp$ and $\nabla G\epp$, we get the existence of approximate solutions. 
\begin{cor} \label{cor:sol}
	Given $\alpha$ as in Remark \ref{remofradius}, there exists a solution $\eta=\eta\epp$ to \eqref{appproblem}-\eqref{appib} in $\mathfrak Q_{T}$ that belongs to the following regularity class:
$$\ei\in W^{1,\infty}\l 0,T;L^2\left(\Omega\right) \r^d\cap AC^2\l [0,T]; AC^2\left(\overline \Omega\right)\r^d,$$
$$\ds \ei\in AC^2\l [0,T]; L^2\left(\Omega\right)\r^d,$$
$$\psi^i:=G\epp (\ds \ei)\in L^\infty\l 0,T; L^2\left(\Omega\right)\r^d,$$
$$\nabla G\epp (\ds \ei)\in L^\infty\l 0,T; L^\infty\left(\Omega\right)\r^d,$$
$$\dt\ei\in L^\infty\l 0,T;L^2\left(\Omega\right) \r^d\cap L^2\l 0,T; H^1\left(\Omega\right)\r^d,$$
$$\ds \psi^i=\ds \l G\epp\left(\ds\ei\right)\r\in L^\infty\l 0,T;L^2\left(\Omega\right) \r^d\cap L^2\l 0,T; H^1\left(\Omega\right)\r^d,$$
$$\dss \ei \in L^\infty\l 0,T;L^2\left(\Omega\right) \r^d.$$ \end{cor}

Note that the norms of the solution $\eta=\eta\epp$ in the corresponding spaces above may depend on $\ep$. At this stage we cannot infer an $L^\infty$ estimate on $\ds \eta$ (even $\ep$-dependent) because we do not control $\ds \ei$ on  $\partial\Omega$. Anyway, we will manage to establish a related bound in Corollary \ref{lemmakappacor}. 

It is straightforward to see that $\eta=\eta\epp$ from Corollary \ref{cor:sol} coincides with the unique solution of the gradient flow  \begin{equation} \dot\eta \in -\partial_{L^2(\Omega; \rm^{3d})}\bm{\mathcal{E}}\epp\left(\eta\right) \label{gradflow}
\end{equation} in the sense of \cite[Theorem 17.2.3]{butatm}, where the driving functional $\bm{\mathcal{E}}\epp$ was defined in \eqref{approxenergy}. This in particular implies that $t\mapsto\bm{\mathcal{E}}\epp\left(\eta(t)\right)$ is a continuous and non-increasing function.


\section{Uniform estimates of the approximate solutions} \label{s:sec4}
In this section we derive various uniform (in $\ep$)  estimates for the approximating solutions $\eiee$ obtained in Corollary \ref{cor:sol}. These bounds are crucial for passing to the limit in Section \ref{s:sec5}. In the sequel, $C$ will always stand for a constant independent of $\ep$. For the sake of readability, we drop the dependence on $\ep$ in the subscripts and write $\ei=\eiee$, $G=G\epp$, $\alpha^{i}=\alpha^{i}_{\ep}$, etc., until the proof of Lemma  \ref{lemmakappa}.

\begin{lem}[Energy estimate]\label{pp1}
	Let $\eta=\{\ei\}$ be a solution of the approximating problem \eqref{appproblem}-\eqref{appib} in $\mathfrak Q_{T}$ as constructed in Corollary \ref{cor:sol}. Then
	
	\begin{equation}
		\bm{\mathcal{E}}\left(\alpha \right)+\su\left(\it|\dt\ei|^2+|\nabla G\left(\ds\ei\right)\dss\ei|^2~dsdt\right)+\|\eta\|^2_{L^\infty\l 0,T;L^1\left(\Omega\right)\r }\leq C.\label{pp12}
	\end{equation}
	Here the constant may only depend on $\alpha$ and $T$, but not on $\ep$. 
\end{lem}
\begin{proof} We first establish a uniform bound (w.r.t. $\ep$)  on the initial energies. Indeed, since $|\ds\ai |=1$,  Remark \ref{rem0} implies that $|G(\ds \ai)|<1$, and using the explicit definition of $Q$ given in \eqref{gdef}, we get that the first terms (for each $i$) in the expansion  $$\bm{\mathcal{E}}\left(\alpha\right)=\su\left(\ii Q\left(\ds\ai\right)ds+\ii\left(-g\right)\cdot\ai ds\right)$$ are uniformly bounded. The second  terms are obviously uniformly bounded.

	We now prove \eqref{pp12}.
	Take the $L^2(\Omega)$-inner product of \eqref{appproblem} and $\dt\ei$ and integrate over $\mathfrak Q_{\mathfrak t}$, $\mathfrak t\in (0,T]$. We obtain
	\begin{equation*}
		\su\itt|\dt\ei|^2~dsdt=\su\itt \ds G\left(\ds\ei\right)\cdot\dt\ei ~dsdt+\su\itt g\cdot\dt\ei ~dsdt
	\end{equation*}
	Now we perform an integration by parts in space in the first member, integrate the second member over time and use the initial conditions, ending up with
	
	\begin{align*}
		\su\itt|\dt\ei|^2~dsdt=&-\su\itt G\left(\ds\ei\right)\cdot\d_{st}\ei ~dsdt+\su\io g\cdot\ei\big|_{t=\tt} ~ds-\su\io g\cdot\eta^{i}\big|_{t=0}~ds\\
		&+\su\int_{0}^{\tt}\underbrace{G\left(\ds\ei\right)\cdot\dt\ei}_{\text{at}~s=1}dt-\su\int_{0}^{\tt}\underbrace{G\left(\ds\ei\right)\cdot\dt\ei}_{\text{at}~s=0}dt\\
		=&-\su\itt G\left(\ds\ei\right)\cdot\d_{st}\ei ~dsdt+\su\io g\cdot\ei(\tt) ~ds-\su\io g\cdot\alpha^{i}~ds\\
		&+\su\int_{0}^{\tt}\underbrace{\underbrace{G\left(\ds\ei(1)\right)\cdot\dt\alpha^i(1)}_{\dt\alpha^i=0}}_{=0}dt-\su\int_{0}^{\tt}\underbrace{\underbrace{G\left(\ds\ei(0)\right)\cdot\dt\bar{\eta}}_{\su G\left(\ds\ei(0)\right)=0}}_{=0}dt.
	\end{align*}  Here $\bar{\eta}(t):=\e(t,0)=\ee(t,0)=\eee\l t,0\r$ denotes the spatial position of the junction. 
	Consequently, 
	\begin{equation} \label{e:prom1}
		\su\itt|\dt\ei|^2~dsdt=-\su\itt G\left(\ds\ei\right)\cdot\d_{st}\ei ~dsdt+\su\io g\cdot\ei(\tt) ~ds-\su\io g\cdot\alpha^{i}~ds.
	\end{equation} For the first term on the right-hand side, we observe that \begin{equation}\label{gder}  G(\ds\ei)\cdot\d_{st}\ei=\dt Q \left(\ds\ei\right),\end{equation} cf. \eqref{compgrad},
where $Q$ is defined as in \eqref{gdef}. 
	In view of \eqref{gder}, \eqref{e:prom1} becomes
	\begin{align*}
		\su\itt|\dt\ei|^2 ~dsdt+\su\io Q &\left(\ds\ei\right)\left(\tt,\cdot\right)+\io\left(-\mathbf g\right)\cdot\eta\left(\tt,\cdot\right)ds\\
		&=\su\io Q \left(\ds\alpha^{i}\left( s\right )\right)ds+\io\left(-\mathbf g\right)\cdot\alpha\left( s\right )ds,
	\end{align*}
	whence\footnote{Of course, equality \eqref{forremarks} is a generic property of gradient flows and at least the fact that its right-hand side is greater than or equal to the left-hand one follows from the general theory, cf. \cite{butatm}. We decided to present a direct and explicit proof here  in order to help the reader to perceive the non-standard boundary conditions of the problem ``by touching''. }
	\begin{equation} \label{forremarks}
		\su\itt|\dt\ei|^2~dsdt+\bm{\mathcal{E}}\left(\eta(\tt)\right)=\bm{\mathcal{E}}(\alpha ).
	\end{equation}
	Using $Q\geq 0$ and the definition of $\bm{\mathcal{E}}$, we derive that
	\begin{equation} \label{bounde}
		\bm{\mathcal{E}}\left(\eta(\tt)\right)\geq-\|\eta(\tt)\|_{L^1\left(\Omega\right)}\|\mathbf g\|_{L^\infty\left(\Omega\right)}.
	\end{equation}
	Hence, employing Jensen's inequality, we can estimate
	\begin{align*}
			\frac 1 3\|\eta(\tt)\|^2_{L^1\left(\Omega\right)}\leq \su\|\eta^i(\tt)\|^2_{L^1\left(\Omega\right)}=\su\l \int_\Omega|\ei(\tt)|~ds\r ^2\\= \su\l \int_\Omega|\ai|~ds+\itt\dt|\ei|~dsdt\r ^2\leq 2\|\alpha\|^2_{L^2(\Omega)}+ 2\su\l \itt|\dt\ei|~dsdt\r ^2\\ \leq 2\|\alpha\|^2_{L^2(\Omega)}+2\tt\su\itt|\dt\ei|^2~dsdt\leq 2\|\alpha\|^2_{L^2(\Omega)}+2T\bm{\mathcal{E}}\left(\alpha\right)+ 2T\|\eta(\tt)\|_{L^1\left(\Omega\right)}\|\mathbf g\|_{L^\infty\left(\Omega\right)}.
	\end{align*}
Simple algebra implies that $\|\eta(\tt)\|_{L^1\left(\Omega\right)}$ is uniformly bounded. Consequently, $\su\it|\dt\ei|^2~dsdt$ is uniformly bounded. 
	On the other hand, from the equality $\ds \l G\left(\ds\ei\right)\r =\dt\ei-g$ we deduce\begin{align*}
		\su\it|\nabla G(\ds \ei)\d_{ss}\ei|^2~dsdt&=\su\it|\ds \l G(\ds\ei)\r|^2 ~dsdt\\&\leq2\su\it|\dt\ei|^2~dsdt+6\it|g|^2~dsdt\leq C.
	\end{align*}
We have used Jensen's inequality and the fact that $\mathbf g=(g,g,g)$. 
\end{proof}
In view of \eqref{bounde} we simultaneously proved the following. 
\begin{cor}\label{lemmaenergy}The energy of the approximating problem $\bm{\mathcal{E}}\left(\eta(t)\right)$ is bounded from below for all $t\in[0,T]$ uniformly in $\ep$.\end{cor}

Since $\eta^i(0)=\alpha^i$ does not depend on $\ep$, the uniform regularity can immediately be improved by the Poincar\'e inequality.

\begin{cor}\label{lemmaeasycor}
	The norm $||\ei||_{L^{\infty}\l 0,T;L^{2}\l\Omega\r\r}$ is uniformly bounded with respect to $\ep$. 
\end{cor}

For the the subsequent family of estimates will need to bound the time away from zero by some constant $\delta>0$.  

\begin{lem}\label{lemmadteta}
	Given $\delta>0$, the norm $||\dt\eta||_{L^{\infty}\l\delta,T;L^2\l\Omega\r\r}$ is bounded uniformly in $\ep$.
\end{lem}

\begin{proof}
	By \cite[Theorem 17.2.3]{butatm}, the right derivative $\dt^+\eta$ exists for all times, and the expression $||\dt^+\eta\l t\r||^2_{L^{2}\l \Omega\r}$ is non-increasing in time. Using \cite[formula (17.79)]{butatm}, we obtain
	\begin{align*}\bm{\mathcal{E}}\left(\alpha \right)-\bm{\mathcal{E}}\left(\eta\l\delta\r\right)\geq &
		\limsup_{h\searrow 0}\bm{\mathcal{E}}\left(\eta\l h\r\right)-\bm{\mathcal{E}}\left(\eta\l\delta\r\right)\\=&\int_{0}^{\delta}||\dt\eta\l t\r||^{2}_{L^{2}\l \Omega\r}dt\\&\geq \int_{0}^{\delta}||\dt^+\eta\l\delta \r||^2_{L^{2}\l \Omega\r}\\&=\delta||\dt^+\eta\l\delta\r||^2_{L^{2}\l\Omega\r}.
	\end{align*}
By \eqref{pp12} and Corollary \ref{lemmaenergy},  the left-hand side is bounded from above uniformly in $\epsilon$.  Hence, $||\dt^+\ei\l\delta\r||_{L^{2}\l\Omega\r}\leq C/\delta$.
  
Since $||\dt^+\eta\l t\r||_{L^{2}\l \Omega\r}$ is non-increasing in time,  we infer that  $||\dt\ei||_{L^{\infty}\l\delta,T;L^{2}\l\Omega\r\r}$=$||\dt^+\ei||_{L^{\infty}\l\delta,T;L^{2}\l\Omega\r\r}$ is bounded uniformly in $\epsilon$. 
\end{proof}
We now derive uniform bounds for $\ki$ that were defined in \eqref{defeq}.  We start with the following lemma.

\begin{lem}\label{lemmaproduct}
	For fixed $\delta>0$, $\ds\psi^{i}$ and the product $|\psi^{i}||\dss\ei-\ep\ds\ki|$  are bounded in $L^{\infty}\l \delta,T;L^{2}\l\Omega\r\r$ uniformly with respect to $\ep$, $i=1,2,3$.
\end{lem}
\begin{proof} By Lemma \ref{lemmadteta}, we know that $||\dt\ei||_{L^{\infty}\l\delta,T;L^2\l \Omega\r\r}\leq C$. Since $\dt\ei=\ds\psi^{i}+g$, we infer that $\ds\psi^{i}$ is bounded in $L^{\infty}\l \delta,T;L^{2}\l\Omega\r\r$ uniformly with respect to $\ep$.
	We differentiate both sides of the equality $$\ds\ei=F\epp(\psi^{i})=\ep\psi^{i}+\frac{\psi^{i}}{\sqrt{\ep+|\psi^{i}|^2}}$$ with respect to $s$ to get 
	$$\dss\ei=\ep\ds\psi^{i}+\frac{\ds\psi^{i}}{\sqrt{\ep+|\psi^{i}|^2}}-\frac{\psi^{i}\l\ds\psi^{i}\cdot\psi^{i}\r}{(\ep+|\psi^{i}|^2)^{3/2}}.$$
	
	We multiply this equality by $\sqrt{\ep+|\psi^{i}|^2}$ and deduce $$\dss\ei\sqrt{\ep+|\psi^{i}|^2}=\ep\ds\psi^{i}\sqrt{\ep+|\psi^{i}|^2}+\ds\psi^{i}-\frac{\psi^{i}  \l \ds \psi^{i}  \cdot \psi^{i}\r}{\ep+|\psi^{i}|^2}.$$ 
	We reorganize the equality above to obtain
	$$\l\dss\ei-\ep\ds\psi^{i}\r\sqrt{\ep+|\psi^{i}|^2}=\ds\psi^{i}-\frac{\psi^{i}\l\ds\psi^{i}\cdot\psi^{i}\r}{\ep+|\psi^{i}|^2}.$$ 
	The right-hand side is bounded in $L^{\infty}\l \delta,T;L^{2}\l\Omega\r\r$ uniformly with respect to $\ep$, hence so is the left-hand side. Consequently, $|\psi^{i}||\dss\ei-\ep\ds\psi^{i}|$ is bounded in $L^{\infty}\l \delta,T;L^{2}\l\Omega\r\r$ uniformly with respect to $\ep$.
\end{proof}

\begin{lem} \label{lemmatre} Let $\Upsilon$ be a finite set in $\mathbb{R}^d$. Assume that there exists a point in the convex hull of $\Upsilon$ such that the distance between it and $\Upsilon$ is greater than or equal to $1$.  Then the radius of the smallest enclosing ball for $\Upsilon$ is greater than or equal to $1$. 
	\end{lem}

\begin{proof} Translating the origin if necessary, we may assume that the origin belongs to the convex hull of $\Upsilon$, and $\Upsilon$ does not intersect with the open unit ball centered in the origin. It suffices to prove that there is no $p\in \mathbb{R}^d$ with $|y-p|<1$ for any $y\in\Upsilon$. Indeed, if such $p$ exists, then $y \cdot p\geq \frac 1 2 |y|^2 - \frac 1 2 |y-p|^2> 0$. Since the origin belongs to the convex hull of $\Upsilon$, we infer $0>0$, a contradiction. 
	\end{proof}

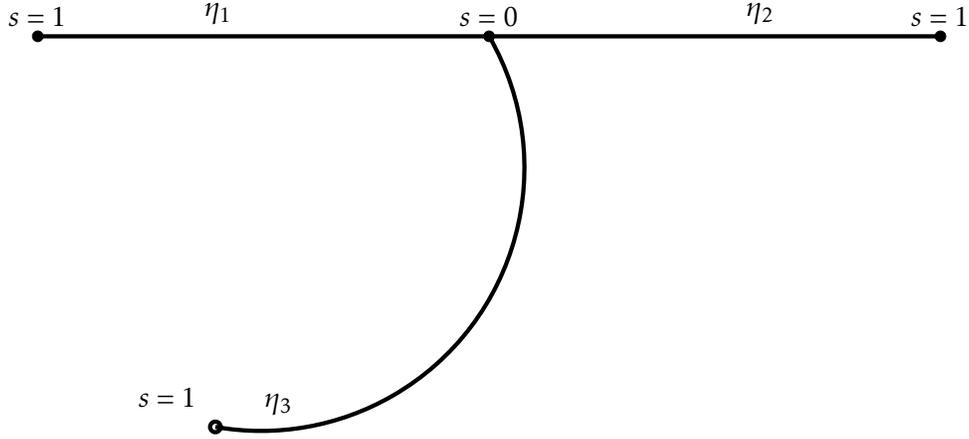
\begin{figure}[h]
	\begin{center}
		\begin{tikzpicture}
			\draw[ultra thick](0,0)--(12,0)node[pos=0.2,sloped,above]{$\eta_{1}$}node[pos=0.8,sloped,above] {$\eta_{2}$} ;
			\draw [ultra thick](6,0) arc (30:-100:3.5cm) circle (2pt) node[anchor=south] {$s=1~~~~~~~~~\eta_{3}$} ;
			\filldraw[black] (6,0) circle (2pt) node[anchor=south] {$s=0$};
			\filldraw[black] (12,0) circle (2pt) node[anchor=south] {$s=1$};
			\filldraw[black] (0,0) circle (2pt) node[anchor=south] {$s=1$};
		\end{tikzpicture}
	\end{center}
	\caption{Symbolic depiction of the 1st scenario in Lemma \ref{lemmakappa}: two arms of the triod tend to the straight position}\label{fig1}
\end{figure}

\begin{figure}[h]
	\begin{center}
		\begin{tikzpicture}
			\draw[ultra thick](-6.80,2)--(0,0)node[pos=0.5,sloped,above]{$\eta_{1}$};
			\draw[ultra thick](0,0)--(6.80,2)node[pos=0.5,sloped,above] {$\eta_{2}$} ;
			\draw [ultra thick](0,0)--(0,-4.3) node[midway,right]{$\eta_{3}$};
			\filldraw[black] (0,0) circle (2pt) node[anchor=south] {$s=0$};
			\filldraw[black] (6.80,2) circle (2pt) node[anchor=south] {$s=1$};
			\filldraw[black] (-6.80,2) circle (2pt) node[anchor=south] {$s=1$};
			\filldraw[black] (0,-4.3) circle (2pt) node[anchor=north] {$s=1$};
		\end{tikzpicture}
	\end{center}
	\caption{Symbolic depiction of the 2nd scenario in Lemma \ref{lemmakappa}: all the arms of the triod tend to the straight position}\label{fig2}
\end{figure}
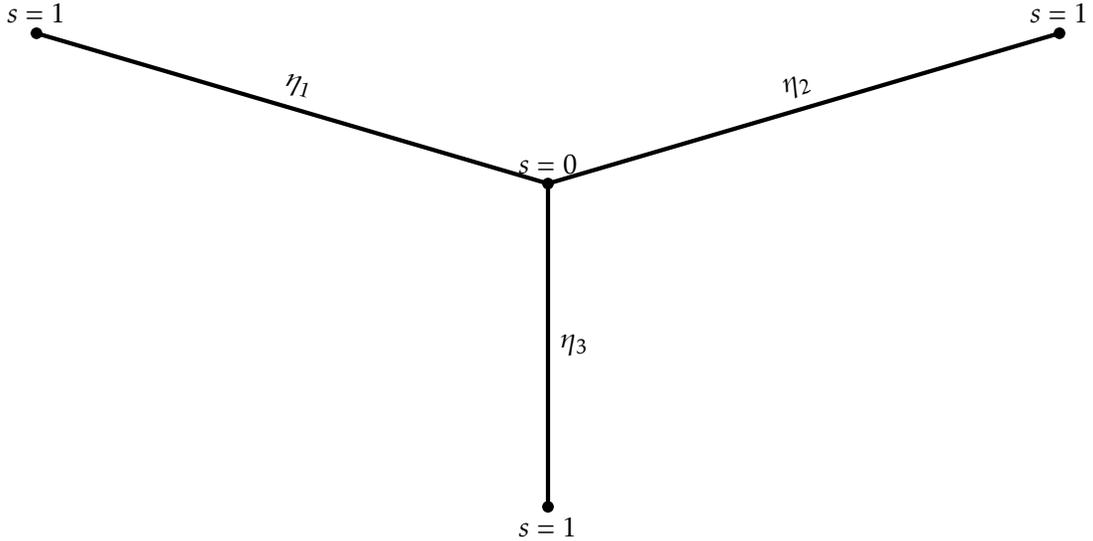
Now we assemble all the ingredients to get the crucial $L^\infty$ bounds for $\psi$ and $\ds \eta$. 
\begin{lem}\label{lemmakappa}
	Given $\delta>0$, the norm $||\psi^{i}||_{L^{\infty}\l\delta,T;L^{\infty}\l\Omega\r\r}$ is uniformly bounded with respect to $\ep$. 
\end{lem}
\begin{proof}
	
	 From now on, we do not omit the subscript $\ep$. However, in this proof we decided to swap the  sub- and superindices for the sake of convenience and readability. 
	
	Step 1. We argue by contradiction. Assume that there is a sequence $\ep^{n}\rightarrow 0$ such that  $$||\psi^{\ep^n}_{1}||_{L^{\infty}\l\delta,T;L^{\infty}\l\Omega\r\r}\to +\infty.$$ Here, without loss of generality, we have chosen the generic $i$ to be equal to $1$.  By the regularity of $\partial_s \psi^{\ep^n}$ and $\dss\eta^{\ep^{n}}$ there exists a set $\mathfrak T_n$ of full measure in $[\delta,T]$ such that $\psi^{\ep^n}_{i}(t)$ and $\eta^{\ep^{n}}_i(t)$ are $C^1$-smooth in $ \overline \Omega$ whereas $\dss\eta^{\ep^{n}}_i(t)\in L^2(\Omega)$ for every $i$ and every $t\in \mathfrak T_n$. Furthermore, by Lemma \ref{lemmaproduct} without loss of generality we can assume that $\ds\psi^{\ep^n}_{i}(t)$ and $|\psi_{i}^{\ep^{n}}\l t,\cdot \r||\dss\eta^{\ep^{n}}_{i}\l t,\cdot\r-\ep^n\ds\psi_{i}^{\ep^{n}}\l t,\cdot \r|$ are bounded in $L^2(\Omega)$ uniformly w.r.t. $n$ and $t\in \mathfrak T_n$. Let $\mathfrak T:= \cap_{n\in \mathbb{N}}\mathfrak T_n$.  Then there is a sequence  $(t^n, s^{n})\in \mathfrak T\times \Omega$ such that $|\psi_{1}^{\ep^n}\l t^{n},s^{n}\r|\rightarrow +\infty$ as $n \to \infty$.   Thus,
	$$\psi_{1}^{\ep^{n}}\l t^n,s\r=\underbrace{\psi_{1}^{\ep^{n}}\l t^n,s^n\r}_{\rightarrow +\infty}+\underbrace{\int_{s^n}^{s}\d_{\xi}\psi^{\ep^{n}}_{1}\l t^{n},\xi\r\d \xi}_{\leq C}$$ when $n \to \infty$. Accordingly, $|\psi_{1}^{\ep^n}(t^n)|\rightarrow +\infty$ uniformly in $s$.  
	
	Step 2. 
	By the boundary conditions,  \begin{equation} \label{bkk} \su \psi_{i}^{\ep^{n}}(t^n,0) =0. \end{equation} 
	By the previous step,  $|\psi_{1}^{\ep^n}(t^n,0)|\rightarrow +\infty$. 
	Hence we have two possible scenarios symbolically pictured in Figures \ref{fig1} and \ref{fig2}, respectively.  The first option is $|\psi_{2}^{\ep^{n}}(t^n,0)|\rightarrow +\infty$ and $|\psi_{3}^{\ep^{n}}(t^n,0)|\leq C$ as $n \to +\infty$ (up to swapping the second and the third arms). The second one is $|\psi_{2}^{\ep^{n}}(t^n,0)|\rightarrow +\infty$ and $|\psi_{3}^{\ep^{n}}(t^n,0)|\rightarrow +\infty$ as $n \to +\infty$. 
	
	Step 3. We start by examining the second scenario. An argument similar to the one of Step 1 shows that $|\psi_{i}^{\ep^n}(t^n)|\rightarrow +\infty$ uniformly in $s$, $i=1,2,3$.  Since $t^n\in \mathfrak T$, we know that $$|\psi_{i}^{\ep^{n}}\l t^{n},\cdot \r||\dss\eta^{\ep^{n}}_{i}\l t^{n},\cdot\r-\ep\ds\psi_{i}^{\ep^{n}}\l t^{n},\cdot \r|$$ is uniformly bounded in $L^{2}\l \Omega\r$. 
	Hence, $$|\dss\eta^{\ep^{n}}_{i}\l t^{n},\cdot\r-\ep^n\ds\psi_{i}^{\ep^{n}}\l t^{n},\cdot \r|\rightarrow 0$$ in $L^{2}\l\Omega\r$ as $n \to +\infty$. On the other hand, $\ds\psi_{i}^{\ep^{n}}\l t^{n}\r $ is uniformly bounded in $L^{2}\l \Omega\r$, whence $$|\ep^n\ds\psi_{i}^{\ep^{n}}\l t^{n} \r|\rightarrow 0$$ in $L^{2}\l \Omega\r$. 
	We conclude that  $|\dss\eta^{\ep^{n}}_{i}\l t^{n},\cdot\r|\rightarrow 0$ in $L^2\l\Omega\r$ as $n \to +\infty$. By Remark \ref{rem0}, $|\psi_{i}^{\ep^{n}}\l t^{n},\cdot \r|\geq 1$ implies $|\ds\eta_{i}^{\ep^{n}}\l t^{n},\cdot \r|\geq 1$ (assuming $n$ to be large enough). 
	
	Step 4. The idea now is to compare the triangle formed by the points $p_i^n:=\eta^{\ep^{n}}_{i}\l t^{n},0\r+\ds \eta^{\ep^{n}}_{i}\l t^{n},0\r$ with the fixed triangle\footnote{Both triangles can be degenerate.} formed by $\eta^{\ep^{n}}_{i}\l t^{n},1\r=\alpha_{i}\l 1\r$, $i=1,2,3$. Observe that  \begin{align*}
		&|\ds\eta_{i}^{\ep^{n}}\l t^{n},0\r-\ds\eta_{i}^{\ep^{n}}\l  t^{n},\xi\r|
		=\left| \int_{0}^{\xi}\dss\eta_{i}^{\ep^{n}}\l  t^{n}\r\, ds\right|\\
		&\leq \int_{0}^{\xi}|\dss\eta_{i}^{\ep^{n}}\l  t^{n}\r|\,ds
		\leq \sqrt{\int_{0}^{1}|\dss\eta_{i}^{\ep^{n}}\l  t^{n}\r|^2\,ds}\rightarrow 0~\text{uniformly}\ \text{in}\ \xi\ \text{as}~n\rightarrow\infty.
	\end{align*}
Hence, \begin{align*}
	&|p_i^n-\alpha_i(1)|= |\eta^{\ep^{n}}_{i}\l t^{n},0\r-\eta^{\ep^{n}}_{i}\l t^{n},1\r+\ds \eta^{\ep^{n}}_{i}\l t^{n},0\r|\\
	&=\Big|\int_{0}^{1} \ds \eta^{\ep^{n}}_{i}\l t^{n},0\r - \ds \eta^{\ep^{n}}_{i}\l t^{n},s\r \, ds \Big|
	\rightarrow 0~\text{as}~n\rightarrow\infty.
\end{align*}
It follows from our assumptions (Remark \ref{remofradius}) that the radius of the smallest enclosing ball of the three points $\alpha_{i}\l 1\r$ is less than $1$. Since the radius of the smallest enclosing ball is a continuous function of the points of a set, it follows that the radius of the smallest enclosing ball of the three points $p_i^n$ is less than $1$ for $n$ sufficiently large. Since the junction point $\eta^{\ep^{n}}_{i}\l t^{n},0\r$ does not depend on $i$, the radius of the smallest enclosing ball of the three points $\tilde p_i^n:= \ds \eta^{\ep^{n}}_{i}\l t^{n},0\r$ is the same as the previous one. By Step 3, $|\tilde p_i^n|\geq 1.$ Moreover, since $\su \psi_{i}^{\ep^{n}}(t^n,0) =0$ and $\tilde p_i^n=F_{\ep^{n}}( \psi_{i}^{\ep^{n}}(t^n,0))$, we conclude that the convex hull of $\{\tilde p_i^n\}$ contains the origin. We arrive at a contradiction because by Lemma \ref{lemmatre} the radius of the smallest enclosing ball of $\{\tilde p_i^n\}$ must be greater than or equal to $1$.

Step 5. We now study the first scenario. Define $p_i^n$ and $\tilde p_i^n$ as in Step 4. The plan is to look at the angle $\theta_n$ between the position vectors of $\tilde p_1^n$ and $\tilde p_2^n$  and to obtain a contradiction from that. 

We first show that $\theta_n$ cannot tend to $\pi$. Indeed, mimicking the arguments of Steps 3 and 4, we can prove that for $i=1,2$ one has $|\ds\eta_{i}^{\ep^{n}}\l t^{n},\cdot \r|\geq 1$ with $n$ large enough,  $|\dss\eta^{\ep^{n}}_{i}\l t^{n},\cdot\r|\rightarrow 0$ in $L^2\l\Omega\r$ and  $$|p_i^n-\alpha_i(1)| \rightarrow 0~\text{as}~n\rightarrow\infty.$$ Hence, $$|\tilde p_1^n-\tilde p_2^n|= |p_1^n-p_2^n|\to |\alpha_1(1)-\alpha_2(1)|<2.$$ Since we have $|\tilde p_1^n|\geq 1 , |\tilde p_2^n|\geq 1$, the angle  $\theta_n$ cannot converge to $\pi$. 

Now take the wedge product of relation \eqref{bkk} with the vector $$\frac {1}{|\ds\eta_{1}^{\ep^{n}}\l t^{n},0 \r||\psi_{2}^{\ep^{n}}(t^n,0) |}\ds\eta_{1}^{\ep^{n}}\l t^{n},0 \r$$ to obtain $$\frac{\psi_{2}^{\ep^{n}}(t^n,0)}{|\psi_{2}^{\ep^{n}}(t^n,0)|} \wedge\frac {\ds\eta_{1}^{\ep^{n}}\l t^{n},0 \r}{|\ds\eta_{1}^{\ep^{n}}\l t^{n},0 \r|}+\frac{\psi_{3}^{\ep^{n}}(t^n,0)}{|\psi_{2}^{\ep^{n}}(t^n,0)|} \wedge\frac {\ds\eta_{1}^{\ep^{n}}\l t^{n},0 \r}{|\ds\eta_{1}^{\ep^{n}}\l t^{n},0 \r|}=0.$$ Since $|\psi_{2}^{\ep^{n}}(t^n,0)|\rightarrow +\infty$ and $|\psi_{3}^{\ep^{n}}(t^n,0)|\leq C$, the second term converges to $0$. Consequently, $$|\sin \theta_n|=\Big|\frac {\ds\eta_{2}^{\ep^{n}}\l t^{n},0 \r}{|\ds\eta_{2}^{\ep^{n}}\l t^{n},0 \r|} \wedge\frac {\ds\eta_{1}^{\ep^{n}}\l t^{n},0 \r}{|\ds\eta_{1}^{\ep^{n}}\l t^{n},0 \r|}\Big|= \Big|\frac{\psi_{2}^{\ep^{n}}(t^n,0)}{|\psi_{2}^{\ep^{n}}(t^n,0)|} \wedge\frac {\ds\eta_{1}^{\ep^{n}}\l t^{n},0 \r}{|\ds\eta_{1}^{\ep^{n}}\l t^{n},0 \r|}\Big|\to 0$$ as $n \to \infty$. 

To obtain a contradiction, it remains to observe that $\theta_n$ cannot tend to $0$. Indeed, taking the scalar product of relation \eqref{bkk} with $$\frac {1}{|\ds\eta_{1}^{\ep^{n}}\l t^{n},0 \r||\psi_{2}^{\ep^{n}}(t^n,0) |}\ds\eta_{1}^{\ep^{n}}\l t^{n},0 \r$$ we get \begin{equation} \frac{\psi_{1}^{\ep^{n}}(t^n,0)}{|\psi_{2}^{\ep^{n}}(t^n,0)|} \cdot\frac {\ds\eta_{1}^{\ep^{n}}\l t^{n},0 \r}{|\ds\eta_{1}^{\ep^{n}}\l t^{n},0 \r|}+\frac{\psi_{2}^{\ep^{n}}(t^n,0)}{|\psi_{2}^{\ep^{n}}(t^n,0)|} \cdot\frac {\ds\eta_{1}^{\ep^{n}}\l t^{n},0 \r}{|\ds\eta_{1}^{\ep^{n}}\l t^{n},0 \r|}+\frac{\psi_{3}^{\ep^{n}}(t^n,0)}{|\psi_{2}^{\ep^{n}}(t^n,0)|} \cdot\frac {\ds\eta_{1}^{\ep^{n}}\l t^{n},0 \r}{|\ds\eta_{1}^{\ep^{n}}\l t^{n},0 \r|}=0.\end{equation} 
The first term is equal to $\frac{\sigma_{1}^{\ep^{n}}}{|\psi_{2}^{\ep^{n}}(t^n,0)||\ds\eta_{1}^{\ep^{n}}\l t^{n},0 \r|}\geq 0$ by \eqref{defeq} and \eqref{sinon}. The third term converges to $0$. Accordingly, the second term, which is equal to $\cos \theta _n$, cannot tend to $1$. 
\end{proof}

\begin{cor}\label{lemmakappacor}
	Given $\delta>0$, the norm $||\ei_{\ep^n}||_{L^{\infty}\l\delta,T;W^{1,\infty}\l\Omega\r\r}$ is uniformly bounded with respect to $\ep$. 
\end{cor}\begin{proof} Since $\ds\ei_{\ep^n}=F_{\ep^n}(\psi_{\ep^n}^{i})$ and the sequence $\{\epsilon_n\}$ is bounded, Lemma \ref{lemmakappa} yields a uniform $L^\infty$ bound for $\ds\ei$. By Lemma \ref{pp1}, $||\ei||_{L^{\infty}\l\delta,T;L^1\l\Omega\r\r}$ is also uniformly bounded with respect to $\ep$, and the claim follows by the mean value theorem. \end{proof}

\begin{lem}\label{sigmaspace}
		Given $\delta>0$, the norm $\|\sigma^{i}_{\ep^n}\|_{L^\infty \l\delta,T;H^{1}\l\Omega \r\r}$ is bounded uniformly in $\ep$.
\end{lem}
\begin{proof}
	In view of Lemma \ref{lemmakappa} and Corollary \ref{lemmakappacor}, the  $L^{\infty}\l\delta,T;L^{\infty}\l \Omega\r\r$-bound for $\sigma$ immediately follows  from the equality $\sigma^{i}=\ds\ei\cdot\psi^{i}$. Differentiating this equality w.r.t. $s$ we obtain $$\ds\sigma^{i}=\ds\psi^{i}\cdot\ds\ei+\psi^{i}\cdot\dss\ei.$$ We estimate the two terms on the right-hand side separately. Firstly, a uniform $L^\infty\l\delta,T;L^{2}\l\Omega \r\r$ bound for $\ds\psi^{i}$  has been already established, cf. Lemma \ref{lemmaproduct}. This together with Corollary \ref{lemmakappacor} implies the uniform boundedness of $\ds\psi^{i}\cdot\ds\ei$ in $L^\infty\l\delta,T;L^{2}\l\Omega \r\r$.
	
	Now, we estimate $\psi^{i}\cdot\dss\ei$. From the explicit expression of $\lambda_{\ep^n}$ in \eqref{lambdaineq}, for $\tau\in\rm^d$ we have 
	\begin{equation}\label{lambdainequality1}
		\lambda_{\ep^n}\l \tau\r=\frac{\sqrt{\ep^n+|G_{\ep^n}\l\tau\r|^2}}{\ep^n\sqrt{\ep^n+|G_{\ep^n}\l\tau\r|^2}+1}\geq \frac{|G_{\ep^n}\l\tau\r|}{\ep^n|G_{\ep^n}\l\tau\r |+1}. 
	\end{equation}
	Thus,
	\begin{align*}
		|G_{\ep^n}\l\ds\ei\r||\dss\ei|&\leq\l\ep^n|G_{\ep^n}\l\ds\ei\r|+1\r|\lambda_{\ep^n}\dss\ei|\\&\leq\l\ep^n|\psi^i|+1\r|\nabla G_{\ep^n}\l\ds\ei\r\dss\ei|. 
	\end{align*}
	By Lemma \ref{lemmakappa}, $|\psi^i|$ is uniformly bounded in $L^{\infty}\l\delta,T;L^{\infty}\l \Omega\r\r$, whence
	\begin{equation*}\label{nablainequality} |\psi^{i}\cdot\dss\ei|\leq |G_{\ep^n}\l\ds\ei\r||\dss\ei|\leq C|\nabla G_{\ep^n}\l\ds\ei\r\dss\ei|=C|\ds \psi^i|.
	\end{equation*}
	Since the right-hand side is uniformly bounded in $L^\infty\l\delta,T;L^{2}\l\Omega \r\r$, so is the left-hand side and, consequently, the spatial derivative $\ds\sigma^i$ itself. 
\end{proof}

\section{Existence of generalized solutions} \label{s:sec5}

We are now at the position to define generalized solutions to the original problem  \eqref{nwgrd}, \eqref{boundary},  \eqref{initial} and to prove their existence. 

\begin{defn}[Generalized solution]\label{def1}
	Given initial data $\alpha^{i}\l s\r\in W^{1,\infty}\l\Omega\r^d$ as in Remark \ref{remofradius}, we call a pair $\l\eta^{i},\sigma^{i}\r$ a generalized solution to \eqref{nwgrd}, \eqref{boundary},  \eqref{initial} in $\mathfrak{Q}_{\infty}$ if\begin{enumerate}[(i)]
		\item \begin{itemize}\item[-]$\eta^{i}\in L^{\infty}_{loc}\l(0,\infty;W^{1,\infty}\l\Omega\r\r^d\cap C_{loc}\l(0,\infty);C\l\overline\Omega\r\r^d\cap AC^2_{loc}\l[0,\infty);L^{2}\l\Omega\r\r^d$, 
			\item[-] $\dt\eta^{i}\in L^{\infty}_{loc}\l\l 0,\infty\r;L^{2}\l\Omega\r\r^d\cap L^{2}_{loc}\l[0,\infty);L^{2}\l\Omega\r\r^d$, 
			\item[-]$\sigma^{i}\in L^{\infty}_{loc}\l\l0,\infty\r;AC^{2}\l\Omega\r\r$,
			\item[-] $\sigma^{i}\ds\eta^{i}\in L^{\infty}_{loc}\l\l 0,\infty\r;AC^{2}\l\Omega\r\r^d$.
			
		\end{itemize}
		\item Each pair $\l\eta^{i},\sigma^{i}\r$ satisfies for a.e.  $\l t,s\r\in \mathfrak{Q}_{\infty}$
		\begin{align}
			&\dt\eta^{i}\l t,s\r=\ds\l\sigma^{i}\l t,s\r\ds\eta^{i}\l t,s\r\r+g,\label{eq35}\\
			&\sigma^{i}\l t,s\r\l|\ds\eta^{i}\l t,s\r|^2-1\r=0,\label{eq36}\\
			&|\ds\eta^{i}\l t,s\r|\leq 1\label{eq37},
		\end{align}
		as well as the initial conditions  $$\eta^{i}\l 0,s\r=\alpha^{i}\l s\r$$ and the boundary conditions\begin{align*}
			&\e\l t,0\r=\ee\l t,0\r=\eee\l t,0\r,\\
			&\eta^{i}\l t,1\r=\alpha\iii\l 1\r,\\
			&\su \sigma^i\l t,0\r\ds\eta^{i}\l t,0\r=0.  
		\end{align*}
		\item The solutions $\eta^{i}$ satisfy the energy dissipation inequality
		\begin{equation}\label{energydiss}
			\su	\int_{\Omega}|\dt\eta^{i}\l t,s\r|^2~ds\leq \su \int_{\Omega}g\cdot\dt\eta^{i}\l t,s\r~ds
		\end{equation}
		for a.e.  $t\in\l 0,\infty\r$.
	\end{enumerate}
	
\end{defn}

\begin{rem}[Discussion of Definition \ref{def1}] \label{nonconv} Note that \eqref{eq36}, \eqref{eq37} is a minor relaxation of the non-convex constraint \begin{equation} \label{e:str} |\ds\eta^{i}\l t,s\r|=1.\end{equation} However, this is not a banal convexification of the constraint since \eqref{eq36} is still not convex. The new constraints \eqref{eq36}, \eqref{eq37} naturally appear from the $(\eta,\sigma,\psi)$-formulation \eqref{nts}, cf. \eqref{eq29} in the proof below. Moreover, if a generalized (in the sense of Definition \ref{def1}) solution $(\eta,\sigma)$  is $C^2$-smooth, then it automatically satisfies the strong constraint \eqref{e:str}. This claim can be shown by following the lines of \cite[Remark 4.2]{shidim1}, \cite[Remark 3.20]{shidim2}. As in \cite{shidim1}, our generalized solutions are, generally speaking, not unique. Yet this has nothing to do with the fact that we slightly relaxed the constraint \eqref{e:str}. As a matter of fact, non-uniqueness can persist even if the strong constraint \eqref{e:str} is imposed, cf. \cite[Remark 6.5]{shidim1}.  Finally, we emphasize that  \eqref{energydiss} is not a direct consequence of \eqref{eq35}, \eqref{eq36}, \eqref{eq37}. \end{rem}

For convenience, we first pass to the limit on finite time intervals. In what follows, we use the shortcut $\mathfrak Q^{*}_{T}:=(\delta,T)\times\Omega$
.  
\begin{prop}\label{ppmain}
	Fix $T>0$ and a small $\delta>0$. Let $\eta\epp$ be a solution to \eqref{appproblem} in $\mathfrak Q_{T}$ with the initial/boundary conditions \eqref{appib} as constructed in Section \ref{evolsec}. Let $\left(\psi^{i},\si\right)$ be defined as in \eqref{defeq}.  Then (up to selecting a subsequence $\ep^n$) there exists a limit $(\eta^i,\sigma^i, \psi^i)$ such that as
$\ep\rightarrow 0$ we have 
	\begin{enumerate}[]
		\item $\ei\epp\rightarrow\ei$ weakly$^{*}$ in $L^{\infty}\left(\delta,T;W^{1,\infty}\left(\Omega\right)\right)^d$, strongly in $C\left(\overline{\mathfrak Q_{T}^{*}}\right)^d$  and weakly in $L^2\left(\mathfrak Q_{T}\right)^d$,
		\item$\dt\ei\epp\rightarrow\dt\ei$ weakly-$*$ in $L^{\infty}\left(\delta,T;L^{2}\left(\Omega\right)\right)^d$ and weakly in $L^2\left(\mathfrak Q_{T}\right)^d$,
		\item $\si\epp\rightarrow\si$ weakly-* in $L^{\infty}\left(\delta,T;H^{1}\left(\Omega\right)\right)$,
		\item $\psi\iii\epp\rightarrow\psi\iii$ weakly-* in $L^{\infty}\left(\delta,T;H^{1}\left(\Omega\right)\right)$.
	\end{enumerate}
	
	The limit satisfies the relation $$\psi\iii=\si\ds\ei\in L^{\infty}\left(\delta,T;H^{1}\left(\Omega\right)\right)$$ 
	and solves \eqref{nwgrd}-\eqref{boundary} in $\mathfrak Q_{T}^{*}$ in the sense that
	\begin{align*}
		&\dt\ei=\ds\left(\si\ds\ei\right)+g\  \text{a.e. in}~\mathfrak Q_{T}^{*},\\
		&\si\left(|\ds\ei|^2-1\right)=0 ~ \text{a.e. in}~\mathfrak Q_{T}^{*},\\
		&\ei\l t,1\r=\alpha\iii\l 1\r,\\
		&\e\l t,0\r=\ee\l t,0\r=\eee\l t,0\r ,\\
		&\su\psi\iii=0~\text{at}~ s=0 ~\text{ for a.e.}~t\in (\delta, T).
	\end{align*}
	\begin{rem} At this stage we don't discuss the validity of the initial condition 
		$\ei\left(0,s\right)=\alpha^{i}\left( s\right )$ that is postponed until Remark \ref{imnit}. \end{rem}
\end{prop}
\begin{proof}
 The weak compactness results for $\ei\epp$, $\si\epp$ and $\psi\iii\epp$ follow immediately from the estimates above. By the Aubin-Lions-Simon theorem,
 \begin{equation}
 	L^{\infty}\l \delta,T;W^{1,\infty}\l\Omega\r\r\cap W^{1,\infty}\l \delta,T;L^{2}\l \Omega\r\r\subset C\l[\delta,T];C\l\overline{\Omega}\r\r
 \end{equation}
 and the embedding is compact, which implies strong compactness of $\ei\epp$ in $C\l[\delta,T];C\l\overline{\Omega}\r\r$.

 Let us show that \begin{equation}
	\psi\iii=\sigma\iii\ds\eta\iii, \quad
	\sigma\iii=\psi\iii\cdot\ds\eta\iii\label{eq29}
\end{equation}
a.e. in $\mathfrak Q^{*}_{T}$.  Since both sides of the equalities \eqref{eq29} are integrable on $\mathfrak Q^{*}_{T}$, it suffices to prove \eqref{eq29} in the sense of the distributions, i.e., that for any $\phi\iii\in L^{2}\l\delta,T;H^1_{0}\l\Omega\r\r$
\begin{align}
	\su\int_{\mathfrak{Q}_{T}^{*}}\psi\iii\phi\iii dsdt&=-\su\int_{\mathfrak{Q}_{T}^{*}}\sigma\iii\eta\iii\ds\phi\iii dsdt-\su\int_{\mathfrak{Q}_{T}^{*}}\ds\sigma\iii\eta\iii\phi\iii dsdt\label{eq30}\\
	\su\int_{\mathfrak{Q}_{T}^{*}}\sigma\iii\phi\iii dsdt&=-\su\int_{\mathfrak{Q}_{T}^{*}}\psi\iii\cdot \eta\iii\ds\phi\iii dsdt-\su\int_{\mathfrak{Q}_{T}^{*}}\ds\psi\iii\cdot\eta\iii\phi\iii dsdt.\label{eq31}
\end{align}
Firstly, applying integration by parts to the equality $\sigma\epp\iii=\psi\epp\iii\cdot\ds\eta\epp\iii$ we obtain $$\su\int_{\mathfrak{Q}_{T}^{*}}\sigma\iii\epp\phi\iii=-\su\int_{\mathfrak{Q}_{T}^{*}}\psi\iii\epp\cdot\eta\iii\epp\ds\phi\iii dsdt-\su\int_{\mathfrak{Q}_{T}^{*}}\ds\psi\iii\epp\cdot\eta\iii\epp\phi\iii dsdt,$$
and due to the strong compactness property of $\{ \eta\epp\}$ given above we can pass to the limit to get \eqref{eq31}.  We now claim
\begin{equation}\label{eq32}
	\lim_{\ep\rightarrow 0}\left| \psi\iii\epp\right| \left| \left| \ds\eta\epp\iii\right|^2-1\right| =0
\end{equation} uniformly in $\mathfrak Q^{*}_{T}$. 
Before proving the claim we show how \eqref{eq30} follows from \eqref{eq32}. Indeed,  with \eqref{eq32} in our hand and noting that $\psi\iii\epp|\ds\eta\epp\iii|^2=\l\psi\epp\iii\cdot\ds\eta\epp\iii\r\ds\eta\epp\iii=\sigma\iii\epp\ds\eta\iii\epp$ we have for each $i=1,2,3$
\begin{equation}\label{eq33}
	\lim_{\ep\rightarrow 0}||\sigma\epp\iii\ds\eta\iii\epp-\psi\iii\epp||_{L^{\infty}\l \mathfrak Q^{*}_{T}\r}=0.
\end{equation}
In particular, for any $\phi\iii\in L^2\l\delta,T;H^1_{0}\l\Omega\r \r$
$$\lim_{\ep\rightarrow 0}\su\int_{\mathfrak{Q}_{T}^{*}}\psi\iii\epp\phi\iii dsdt=\lim_{\ep\rightarrow 0}\su\int_{\mathfrak{Q}_{T}^{*}}\sigma\epp\iii\ds\eta\epp\iii\phi\iii dsdt.$$
An integration by parts applied to the integral on the right-hand side gives $$\lim_{\ep\rightarrow 0}\su\int_{\mathfrak{Q}_{T}^{*}}\psi\epp\iii\phi\iii dsdt=\lim_{\epsilon\rightarrow 0}\su\int_{\mathfrak{Q}_{T}^{*}}\l-\sigma\epp\iii\eta\epp\iii\ds\phi\epp\iii-\ds\sigma\epp\iii\eta\iii\epp\phi\iii\r dsdt.$$ This together with the compactness properties established above yields \eqref{eq30}.

We now provide a proof of \eqref{eq32}. By the definition of $F\epp$ in \eqref{dff},
\begin{align*}
	|\ds\eta\iii\epp|-1&=\left|F\epp\l\psi\iii\epp\r\right|-1\\&=\ep|\psi\epp\iii|+\frac{|\psi\iii\epp|}{\sqrt{\ep+|\psi\epp\iii|^2}}-1\\&=\ep|\psi\epp\iii|-\frac{\ep}{\sqrt{\ep+|\psi\iii\epp|^2}\l\sqrt{\ep+|\psi\iii\epp|^2}+|\psi\iii\epp|\r}.
\end{align*}
Thus,
\begin{align*}
	|\psi\iii\epp|\left||\ds\eta\iii\epp|^2-1 \right|&=\left||\ds\eta\iii\epp|+1 \right|\left|\ep|\psi\epp\iii|^2-\frac{\ep|\psi\iii\epp|}{\sqrt{\ep+|\psi\iii\epp|^2}\l\sqrt{\ep+|\psi\iii\epp|^2}+|\psi\iii\epp|\r} \right|\\&\leq\left||\ds\eta\iii\epp|+1 \right|\left(\ep|\psi\epp\iii|^2+\sqrt{\epsilon}\right).
\end{align*}
This together with uniform $L^\infty$ bounds on $\ds\ei\epp$ and $\psi\epp \iii$ yields
 \eqref{eq32}.

Passing to the limit in $L^2\l \mathfrak{Q}_{T}^{*}\r$ in  $\dt\eta\iii\epp=\ds\psi\iii\epp+g$ and using \eqref{eq29}, we obtain $\dt\eta\iii=\ds\l\sigma\iii\ds\eta\iii\r+g$. Moreover, by \eqref{eq29}, $$\sigma\iii\l|\ds\eta\iii|^2-1\r= \psi\iii\cdot \ds\eta\iii-\sigma\iii =0.$$  

 Due to the strong uniform convergence of $\eta^i\epp$, 
we have $\alpha\iii\l 1\r=\eta\epp\iii\l t,1 \r\rightarrow\eta\iii\l t,1 \r$, whence $\eta\iii\l t,1 \r=\alpha\iii\l 1\r$ for all $t\in[\delta,T]$.  The condition $\e\epp\l t,0\r=\ee\epp\l t,0\r=\eee\epp\l t,0\r$ similarly passes to the limit. To check the validity of the boundary condition at $s=0$ for $\psi$, we swap the variables $t$ and $s$, noting that $\psi^i \epp$ are uniformly bounded and weakly-* converging in $H^1(0,1; L^\infty(\delta,T))$. Employing, for instance, \cite[Corollary 2.2.1]{ZV08}, we get \begin{equation}  
	\label{e:emb2} H^1(0,1; L^\infty(\delta,T))= AC^2([0,1]; L^\infty(\delta,T)).\end{equation} 
Hence, by the Aubin-Lions-Simon theorem, the embedding \begin{equation*}   H^1(0,1; L^\infty(\delta,T))\subset C([0,1]; H^{-1}(\delta,T))\end{equation*} 
is compact, whence we may assume that $\psi^i \epp\to \psi\iii$ strongly in $C([0,1]; H^{-1}(\delta,T))$. Thus, $$0=\su \psi\epp\iii(\cdot,0)\to\su \psi\iii(\cdot,0)$$ in $H^{-1}(\delta,T)$.  Due to \eqref{e:emb2},  $\su \psi\iii(\cdot,0)=0$ in $L^\infty(\delta,T)$. 
\end{proof}

\begin{rem}[Initial conditions] \label{imnit} By the Aubin-Lions-Simon theorem, the embedding \begin{equation*}   H^1(0,T; L^2(\Omega))\subset C([0,T]; H^{-1}(\Omega))\end{equation*}  is compact. 
	Since $\eta\epp\iii$ (w.l.o.g.) converge weakly in $H^1(0,T; L^2(\Omega))$ we can pass to the limit in the initial conditions to obtain 	$\ei\left(0,\cdot\right)=\alpha^{i}$ in $H^{-1}(\Omega)$. However, since $H^1(0,T; L^2(\Omega))=AC^2(0,T; L^2(\Omega))$, the initial conditions actually hold in $L^2(\Omega)$.  \end{rem} 

\begin{prop}\label{pp15}
	Let $\left(\ei,\si\right)$ be the limiting solution obtained in Proposition \ref{ppmain}. Then  \begin{enumerate}[(i)]
		\item $|\ds\ei\left(t,s\right)|\leq 1 $ for a.e. $\left(t,s\right)\in\mathfrak Q^*_{T}$;
		\item $\si\geq 0$ for a.e. $\left(t,s\right)\in \mathfrak Q^*_{T}$;
		\item \eqref{energydiss} holds for a.a. $t\in (\delta, T)$. 
	\end{enumerate}
\end{prop}

We omit the proof since it follows the same lines as the proofs of \cite[Proposition 3.6 and Theorem 3]{shidim1}.

Employing a diagonal argument and taking into account Proposition \ref{pp15} and Remark \ref{imnit}, it is easy to deduce Theorem \ref{thm17} from Proposition \ref{ppmain}.

\begin{rem}[Single cord with two fixed ends] \label{remst} The results of the paper, mutatis mutandis, are valid for the overdamped fall of a single inextensible string with the ends fixed at two distinct spatial points (it suffices to observe that such a string can be viewed as a degenerate ``triod'' with one arm having zero length); remember that \cite{shidim1} studied the case of one free and one fixed end (i.e., a ``whip''). More precisely, we have the following result. \end{rem}

\begin{prop}\label{corst}
	Given $\alpha\l s\r\in W^{1,\infty}\l\Omega\r^d$ satisfying $|\alpha(0)-\alpha(1)|<1$, $|\ds \alpha(s)|=1$ a.e. in $\Omega$, there exists a generalized solution to \begin{equation}\label{nsingle}
		\begin{cases}
			\d_{t}\eta=\d_{s}\left(\sigma\d_{s}\eta\right)+g,\\
			|\ds\eta|=1,\\
			\eta\l t,0\r=\alpha\l 0\r,\quad \eta\l t,1\r=\alpha\l 1\r,\\
			\eta\l 0,s\r=\alpha\l s\r.
		\end{cases} 
	\end{equation}  in $\mathfrak{Q}_{\infty}$. Moreover, $\sigma\l t,s\r\geq 0 $ for almost every $\l t,s\r\in \mathfrak{Q}_{\infty}$.
\end{prop}
\begin{rem}[Exponential decay] \label{expd} Employing energy methods, the authors of \cite{shidim1} established exponential decay of the generalized solution towards the equilibrium state (the whip hanging downwards) in the case of one free and one fixed end. The proof was heavily relying on Hardy's inequality. That strategy cannot be easily adapted even to the string with two fixed ends (the corresponding steady state is the catenary) since that would require to substitute Hardy's inequality with a troublesome non-standard variant of the Poincar\'e inequality. Moreover, to the best of our knowledge, the explicit characterization of the steady state with the least potential energy is not available for the triod. \end{rem}

\subsection*{Acknowledgment}
AT and DV are partially supported by CMUC-UID/MAT/00324/2020, funded by the Portuguese Government through FCT/MCTES and co-funded by the ERDF through the Partnership Agreement PT2020. AT is partially supported by the FCT grant PD/BD/150352/2019.

\def\cprime{$'$}

\appendix
\label{apa1}

\section{The gradient of the potential energy}
Here we present a formal computation of $-\nabla_{\mathcal A} E$. Fix a reference network $\eta\in \mathcal A$.  Firstly, it is clear that $-\nabla_{L^{2}} E\l\eta\r=g$. Hence, by some basic Riemannian geometry \cite{docarmo},
\begin{equation}\label{projectionequation}
-\nabla_{\mathcal A} E\l\eta\r=P_{\eta}g,
\end{equation}
where $P_{\eta}g$ is the orthogonal projection of $g$ onto the tangent space $T_{\eta}\mathcal A$ that was defined in \eqref{tangentspace}. Assume that the following system of ODE
\begin{equation} \d_{ss}\sigma^{i}-|\d_{ss}\ei|\sigma^{i}=0\end{equation} with the initial/terminal conditions
\begin{align}
\su \sigma^{i}\ds\ei=0~\text{at}~s=0,\label{eqpoint}\\
\ds\sigma^{i}\ds\ei+\sigma^{i}\d_{ss}\ei+g ~\text{does not depend on}~i~\text{at}~s=0,\label{eq52a}\\
\ds\sigma^{i}\ds\ei+\sigma^{i}\d_{ss}\ei=-g~\text{at}~s=1\label{eq15},
\end{align}
is solvable for $\sigma\iii$, $i=1,2,3$. We claim that 
\begin{equation}
P_{\eta} g:=\l g+\ds\l\sigma^{1}\ds\e\r,g+\ds\l\sigma^{2}\ds\ee\r,g+\ds\l\sigma^{3}\ds\eee\r\r
\end{equation}
fulfills the relevant conditions for the image of $g$ under the orthogonal projection, namely, $P_{\eta}g\in T_{\eta} g$ and $g-P_{\eta}g$ is $L^{2}$-orthogonal to any $\upsilon^{i}\in T_{\eta}\mathcal A$. Indeed, differentiating the constraints $|\ds\ei|^2=1$ we find
\begin{align*}
\ds\ei\cdot\dss\ei=0,\\
\ds\ei\cdot\d_{sss}\ei=-|\dss\ei|^2.
\end{align*} 
Hence, $$\ds\l P_{\eta} g\r\cdot\ds\eta=\dss\sigma^{i}-|\dss\ei|^2\sigma^{i}=0.$$
It is easy to see that we have $P_{\eta}g\l 1\r=0$ by \eqref{eq15} componentwise. Moreover, by \eqref{eq52a}, $(P_{\eta}g)^{i}\l 0\r$ does not depend on $i$.  We have proved that  $P_{\eta}g\in T_{\eta}\mathcal A$. Finally, for any $\upsilon^{i}\in T_{\eta}\mathcal A$, we obtain after summing the terms and integration by parts
\begin{align*}
\su\ii \l g-\l P_{\eta}g\r\iii\r\cdot \upsilon^{i}ds&=\su\ii-\ds\l\sigma\iii\ds\ei\r\cdot\upsilon\iii ds\\ 
&=\su\ii\sigma\iii\ds\ei\cdot\ds\upsilon\iii ds-\su\sigma\iii\ds\ei\cdot\upsilon\iii\Big|_{s=0}^{s=1}.
\end{align*} 
It follows from the definition of $T_{\eta}\mathcal A$ in \eqref{tangentspace} and equality \eqref{eqpoint} that both terms vanish.

This formally justifies that the PDE form of the gradient flow \eqref{gradequality} is  \eqref{nwgrd}, \eqref{boundary}. A similar but slightly amended argument formally implies that \eqref{newton} is equivalent to \eqref{triodeq}, \eqref{boundary}. 
\vspace{1cm}
\end{document}